\documentclass{amsart}
\usepackage{amsmath,amssymb,amsfonts}
\usepackage{mathtools}
\usepackage{esvect}
\usepackage[shortlabels]{enumitem}
\setlist{leftmargin=*}

\usepackage{amsfonts}

\usepackage[colorlinks=true, linkcolor=blue]{hyperref}

\usepackage[sorted]{amsrefs}

\theoremstyle{plain}
\newtheorem{theorem}{Theorem}[section]
\newtheorem{prop}[theorem]{Proposition}
\newtheorem{fact}[theorem]{Fact}
\newtheorem{lemma}[theorem]{Lemma}
\newtheorem{cor}[theorem]{Corollary}

\theoremstyle{definition}
\newtheorem{defn}[theorem]{Definition}
\newtheorem{remark}[theorem]{Remark}

\newtheorem{problem}[theorem]{Problem}
\newtheorem{expl}{Example}
\newtheorem{conj}{Conjecture}

\def\Th{\operatorname{Th}}
\def\tp{\operatorname{tp}}
\def\EM{\operatorname{EM}}

\global\long\def\Aut{\operatorname{Aut}}

\global\long\def\M{\operatorname{\mathbb{M}}}

\global\long\def\NTP{\operatorname{NTP}}

\global\long\def\TP{\operatorname{TP}}

\global\long\def\tp{\operatorname{tp}}

\global\long\def\ded{\operatorname{ded}}

\global\long\def\opg{\operatorname{opg}}

\global\long\def\qftp{\operatorname{qftp}}

\global\long\def\cof{\operatorname{cf}}

\global\long\def\Th{\operatorname{Th}}

\global\long\def\op{\operatorname{op}}

\title{Mekler's construction and generalized stability}
\author{Artem Chernikov and Nadja Hempel}

\begin{document}
\begin{abstract}
Mekler's construction gives an interpretation of any structure in a finite relational language in a  group (nilpotent of class 2 and exponent $p>2$, but not finitely generated in general). Even though this construction is not a bi-interpretation, it is known to preserve some model-theoretic tameness properties of the original structure including stability and simplicity. We demonstrate that $k$-dependence of the theory is preserved, for all $k \in \mathbb{N}$, and that $\NTP_2$ is preserved. We apply this result to obtain first examples of strictly $k$-dependent groups (with no additional structure).
\end{abstract}

\maketitle

\section{Introduction}
Mekler's construction \cite{mekler1981stability} provides a general method to interpret any structure in a finite relational language in a pure $2$-nilpotent group of finite exponent (the resulting group is typically not finitely generated). This is not a bi-interpretation, however it tends to preserve various model-theoretic tameness properties. First Mekler proved that for any cardinal $\kappa$ the constructed group is $\kappa$-stable if and only if the initial structure was \cite{mekler1981stability}. Afterwards,  it was shown by Baudisch and Pentzel that simplicity of the theory is preserved, and by Baudisch that, assuming stability, CM-triviality is also preserved \cite{baudisch2002mekler}. See \cite[Section A.3]{hodges1993model} for a detailed exposition of Mekler's construction.



The aim of this paper is to investigate further preservation of various generalized stability-theoretic properties from Shelah's classification program \cite{shelah1990classification}. We concentrate on the classes of $k$-dependent and $\NTP_2$ theories.

The classes of $k$-dependent theories (see Definition \ref{def: k-dependence}), for each $k \in \mathbb{N}$,  were defined by Shelah in \cite{shelah2014strongly}, and give a generalization of the class of NIP theories (which corresponds to the case $k=1$). See \cite{shelah2007definable, hempel2016n, chernikov2014n} for some further results about $k$-dependent groups and fields and connections to combinatorics. In Theorem \ref{thm: k-dependence is preserved} we show that Mekler's construction preserves $k$-dependence.
Our initial motivation was to obtain algebraic examples that witness the strictness of the $k$-dependence hierarchy. For $k \geq 2$, we will say that a theory is \emph{strictly $k$-dependent} if it is $k$-dependent, but not $(k-1)$-dependent. The usual combinatorial example of a strictly $k$-dependent theory is given by the random $k$-hypergraph. The first example of a strictly $2$-dependent group was given in \cite{hempel2016n} (it was also considered in \cite[Example 4.1.14]{wagner2002simple}):

\begin{expl}\label{ex: extraspecial}
Let $G$ be $\oplus_{\omega}\mathbb F_p$, where $\mathbb F_p$ is the finite field with $p$ elements. Consider the structure 
$\mathcal{G} = (G, \mathbb F_p, 0, +,\cdot)$, where $0$ is the neutral element, $+$ is addition in $G$, and $\cdot$ is the bilinear form $(a_i)_i \cdot (b_i)_i = \sum_i a_i b_i$ from $G$ to $\mathbb F_p$. This structure is not NIP, but is $2$-dependent. In the case $p=2$, $\mathcal{G}$ is interpretable in an extra-special $p$-group $\mathcal{H} = (H,\cdot, 1)$, and conversely $\mathcal{H}$ is interpretable in $\mathcal{G}$ (see \cite[Proposition 3.11]{macpherson2008one} and the discussion around it, or the appendix in \cite{millietdefinable}). Hence $\mathcal{H}$ provides an example of a strictly $2$-dependent pure group.

 \end{expl} 

In Corollary \ref{cor: strictly k-dep groups} we use Mekler's construction to show that for every $k$, there is a strictly $k$-dependent pure group.

The class of $\NTP_2$ theories was defined in \cite{shelah1980simple} (see Definition \ref{def: NTP2}). It gives a common generalization of simple and NIP theories (along with containing many new important examples), and more recently it was studied in e.g. \cite{chernikov2012forking, chernikov2014theories, yaacov2014independence, chernikov2015groups}. In Theorem \ref{thm: NTP2 is preserved} we show that Mekler's construction preserves $\NTP_2$.

The paper is organized as follows. In Section \ref{sec: Mekler} we review Mekler's construction and record some auxiliary lemmas, including the key lemma about type-definability of partial transversals and related objects (Proposition \ref{prop: type def}). In Section \ref{sec: NIP} we prove that NIP is preserved. In Section \ref{sec: kDep} we discuss indiscernible witnesses for $k$-dependence and give a proof that Mekler's construction preserves $k$-dependence. As an application, for each $k\geq 2$ we construct a strictly $k$-dependent pure group 
and discuss some related open problems.
Finally, in Section \ref{sec: NTP2} we prove that Mekler's construction preserves $\NTP_2$.

\subsection*{Acknowledgements}
We would like to thank the anonymous referee for a very detailed report and many useful suggestions on improving the paper. We also thank JinHoo Ahn for pointing out some typos in the preliminary version.

Both authors were partially supported by the NSF Research Grant DMS-1600796, by the NSF CAREER grant DMS-1651321 and by an Alfred P. Sloan Fellowship.

\section{Preliminaries on Mekler's construction}\label{sec: Mekler}

We review Mekler's construction from \cite{mekler1981stability}, following the exposition and notation in \cite[Section 
A.3]{hodges1993model} (to which we refer the reader for further details).

\begin{defn}
A graph (binary, symmetric relation without self-loops) is called \emph{nice} if it satisfies the following two properties:
\begin{enumerate}
\item there are at least two vertices, and for any two distinct vertices $a$ and $b$ there is some vertex $c$ different from $a$ and $b$ such that $c$ is joined to $a$ but not to $b$;
\item there are no triangles or squares in the graph.
\end{enumerate}
\end{defn}

For any graph $C$ and an odd prime $p$, we define a $2$-nilpotent group of exponent $p$ denoted  by $G(C)$ which is generated freely in the variety of $2$-nilpotent groups of exponent $p$ by the vertices of $C$ by imposing that two generators commute if and only if they are connected by an edge in $C$.

Now, let $C$ be a nice graph and consider the group $G(C)$. Let $G$ be any model of $\Th(G(C))$.  We consider the following $\emptyset$-definable equivalence relations on the elements of $G$. 

\begin{defn}
Let $g$ and $h$ be elements of $G$, then

\begin{itemize}
\item $g \sim h$, if $C_G(g) = C_G(h)$.
\item $g \approx h$ if there is some  natural number $r$  and $c$ in $Z(G)$ such that $g = h^r \cdot c$.
\item $g \equiv_Z h$ if $g\cdot Z(G)= h\cdot  Z(G)$.
\end{itemize}

\end{defn}

Note that  $g \equiv_Z h$ implies $g \approx h$, which implies $g \sim h$.

\begin{defn}
Let $g$ be an element of $G$ and let $q$ be a natural number. We say that $g$ is \emph{of type $q$} if there are $q$ different $\approx$-equivalence classes in the $\sim$-class $[g]_{\sim}$ of $g$. Moreover, we say that $g$ is \emph{isolated} if all non central $h \in G$ which commute with $g$ are $\approx$-equivalent to $g$.
\end{defn}

All non-central elements of $G$ can be partitioned into four different $\emptyset$-definable classes (see \cite[Lemma A.3.6 - A.3.10]{hodges1993model} for the details): 
\begin{enumerate}

\item elements of type $1$ which are not isolated, also referred to as \emph{elements of type $1^{\nu}$} (in $G(C)$ this class includes the elements given by the vertices of $C$), 
\item elements of type $1$ which are isolated, also referred to as \emph{elements of type  $1^{\iota}$},
\item elements of type $p$, and
\item elements of type $p-1$.
\end{enumerate}
The elements of the latter two types are always non-isolated (it is easy to see from the definition that only an element of type $1$ can be isolated).

By \cite[Lemma A.3.8, (a) $\Leftrightarrow$ (b)]{hodges1993model}, for every element $g \in G$ of type $p$, the non-central elements of $G$ which commute with $g$ are precisely the elements $\sim$-equivalent to $g$, and an element $b$ of type $1^{\nu}$ together with the elements $\sim$-equivalent to $b$.

\begin{defn}\label{def: handle}
For every element $g \in G$ of type $p$, we call an element $b$ of type $1^{\nu}$ which commutes with $g$ a \emph{handle of $g$}. 
\end{defn}

\begin{fact}\label{fact: handle}
	By the above, we obtain immediately that a handle is definable from $g$ up to $\sim$-equivalence.
	\end{fact}

Note here, that the center of $G$ as well as the quotient $G/Z(G)$ are elementary abelian $p$-groups. Hence they can be viewed as $\mathbb F_p$-vector spaces. From now on, \emph{independence} over some supergroup of $Z(G)$ will refer to linear independence in terms of the corresponding $\mathbb F_p$-vector space.


\begin{defn}\label{def: transversal}
Let $G$ be a model of $\Th(G(C))$. We define the following:
\begin{itemize}
\item A \emph{$1^{\nu}$-transversal} of $G$ is a set $X^{\nu}$ consisting of one representative for each $\sim$-class of elements of type $1^{\nu}$ in $G$.
\item An element is \emph{proper} if it is not a product of any elements of type $1^{\nu}$ in $G$.
\item A \emph{$p$-transversal} of $G$ is a set $X^p$ of pairwise $\sim$-inequivalent proper elements of type $p$ in $G$ which is maximal with the property that if $Y$ is a finite subset of $X^p$ and all elements of $Y$ have the same handle, then $Y$ is independent modulo the subgroup generated by all elements of type $1^{\nu}$ in $G$ and $Z(G)$.
\item A \emph{$1^{\iota}$-transversal} of $G$ is a set $X^{\iota}$  of representatives of $\sim$-classes of proper elements of type $1^{\iota}$ in $G$ which is maximal independent modulo the subgroup generated by all elements of types $1^{\nu}$ and $p$ in $G$, together with $Z(G)$.
\item A set $X \subseteq G$ is a \emph{transversal of $G$} if $X= X^{\nu} \sqcup X^p \sqcup X^{\iota}$, where $X^{\nu}, X^p$ and $X^{\iota}$ are some transversals of the corresponding types.
\item A subset $Y$ of a transversal is called a \emph{partial transversal} if it is closed under handles (i.\ e.\ for any element $a$ of type $p$ in $Y$ there is an element of type $1^{\nu}$ in $Y$ which is a handle of $a$).
\item For a given (partial) transversal $X$, we denote by $X^{\nu}$, $ X^p$, and $X^{\iota}$ the elements in $X$ of the corresponding types.
\end{itemize}
\end{defn}

\begin{lemma}\label{lem: transversal type def}
	Let $G \models \Th(G(C))$. Given a small tuple of variables $\bar{x} = \bar x^{\nu \frown} \bar{x} ^{p \frown} \bar{x}^{\iota}$, there is a partial type $\Phi(\bar{x})$ such that for any tuples $\bar{a}^\nu, \bar{a}^p$ and $\bar{a}^{\iota}$ in $G$,  we have that $G \models \Phi(\bar{a}^\nu, \bar{a}^p, \bar{a}^{\iota})$ if and only if every element in  $\bar{a}^\nu, \bar{a}^p$ and $\bar{a}^{\iota}$ is of type $1^{\nu}, p$ and $1^{\iota}$, respectively, and $\bar{a} = \bar{a} ^{\nu \frown} \bar{a} ^{p \frown} \bar{a}^{\iota}$ can be extended to a transversal of $G$.
	\end{lemma}
\begin{proof}
By inspecting Definition \ref{def: transversal}. For example, let's describe the partial type $\Phi^p(\bar{x}^p)$ expressing that $\bar{x}^p = (x_i^p : i < \kappa)$ can be extended to a $p$-transversal (the conditions on $\bar{x}^\nu$ and $\bar{x}^\iota$ are expressed similarly). For $q \in \mathbb{N}$, let $\phi_q(x)$ be the formula defining the set of all elements of type $q$ in $G$, let $\phi_\iota(x)$ define the set of isolated elements, and let $\phi_{\textrm{h}}(x_1, \ldots, x_q)$ express that $x_1, \ldots, x_q$ have the same handle.
The set of proper elements is defined by 
$$\Phi_{\textrm{prop}}(x) := \{ \forall y_1 \ldots \forall y_{n-1} (\bigwedge_{i<n} \phi_1(y_i) \land \neg \phi_\iota(y_i)) \rightarrow  (x \neq \prod_{i<n} y_i ) : n \in \mathbb{N} \}.$$
The independence requirement in the definition of a $p$-transversal can be expressed as a partial type 
$\Phi_{\textrm{ind}}(\bar{x}^p)$ containing the formula

$$ \phi_{\textrm{h}}(x^p_{i_0}, \ldots, x^p_{i_{n-1}}) \rightarrow $$
$$\left( \forall y_0 \ldots \forall y_{m-1} \forall z \bigwedge_{i<m}(\phi_1(y_i) \land \neg \phi_\iota(y_i) \land z \in Z(G)) \rightarrow z \cdot \prod_{i<m} y_i^{r_i} \neq \prod_{j<n} (x^p_{i_j})^{q_j}  \right)$$ 

for each $m, n \in \mathbb{N}$, $i_0 < \ldots < i_{n-1} < \kappa$ and $q_0, \ldots, q_{n-1}, r_0, \ldots, r_{m-1} \in \mathbb{N}, 0 < q_j <p$.

 Then we can take $\Phi^p(\bar{x}^p) :=  \{ \phi_p(x_i^p) : i < \kappa\} \cup \{ \neg( x_i^p \sim x_j^p) : i<j<\kappa \} \cup \bigcup_{i < \kappa} \Phi_{\textrm{prop}}(x_i^p) \cup \Phi_{\textrm{ind}}(\bar{x}^p)$.

\end{proof}

The following can be easily deduced from \cite[Corollary A.3.11]{hodges1993model}.

\begin{fact}
Let $C$ be a nice graph. 
There is an interpretation $\Gamma$ such that for any model $G$ of $\Th(G(C))$, we have that $\Gamma (G)$ is a model of $\Th(C)$. More specifically, the graph $\Gamma(G) = (V,R)$ is given by the set of vertices $V=\{ g \in G : g \textrm{ is of type }1^\nu, g \notin Z(G) \} / \approx$ and the (well-defined) edge relation $([g]_{\approx}, [h]_{\approx}) \in R$ $\iff$ $g,h$ commute.
\end{fact}

The full set of transversal gives another important graph, a so called cover of a nice graph, which we define below.

\begin{defn}\label{def: cover} \begin{enumerate}
\item
	Let $C$ be an infinite nice graph. A graph $C^+$ containing $C$ as a subgraph is called a \emph{cover} of $C$ if for every vertex $b \in C^+ \setminus C$, one of the following two statements holds:
	\begin{itemize}
		\item there is a unique vertex $a$ in $C^+$ that is joined to $b$, and  moreover $a$ is in $C$ and it has infinitely many adjacent vertices in $C$;
		 
		 \item $b$ is joined to no vertex in $C^+$.
		\end{itemize}
	
\item	A cover $C^+$ of $C$ is a \emph{$\lambda$-cover} if 
	\begin{itemize}
		\item for every vertex $a$ in $C$ the number of vertices in $C^+\setminus C$ joined to $a$ is $\lambda$ if $a$ is joined to infinitely many vertices in $C$, and zero otherwise;
		\item the number of new vertices in $C^+ \setminus C$ which are not joined to any other vertex in $C^+$ is $\lambda$.
		\end{itemize}
	\end{enumerate}

\end{defn}
Observe that a proper cover of a nice graph is never a nice graph.

\begin{fact}\cite[Corollary A.3.14(a) + proof of (c)]{hodges1993model}. \label{rem: X as a cover}
Given a $1^\nu$-transversal $X^\nu$ of $G$, we identify  the elements of $X^\nu$ with the set of vertices of $\Gamma(G)$ by mapping $x \in X^\nu$ to its class $[x]_{\approx}$. Then a transversal $X$ can be viewed as a cover of the nice graph given by the elements of type $1^\nu$ in $X$, with the edge relation given by commutation. (See  Remark \ref{rem: cover explanation} for an explanation.)
\end{fact}

\begin{fact}\cite[Theorem A.3.14, Corollary A.3.15]{hodges1993model}\label{fac: decomposition}
Let $G$ be a model of $\Th(G(C))$ and let $X$ be a transversal of $G$. 
\begin{enumerate}
\item There is a subgroup of $Z(G)$ which we denote by $K_X$ such that $G = \langle X \rangle \times K_X$.
\item The group $K_X$ is an elementary abelian $p$-group, in particular $\Th(K_X)$ is stable and eliminates quantifiers.
\item If $G$ is saturated and uncountable, then both the graph $\Gamma(G)$ and the group $K_X$ are also saturated (as $|K_X| = |G|$ and $\Th(K_X)$ is uncountably categorical).
\item If $G$ is a saturated model of $\Th(G(C))$, then every automorphism of $\Gamma(G)$ can be lifted to an automorphism of $G$.
\item $\langle X \rangle \cong G(X)$ via an isomorphism which is the identity on the elements in $X$ (where $X$ is viewed as a graph as in Remark \ref{rem: X as a cover}, note that it doesn't have to be nice).
\item If $G$ is saturated then $X$ (where $X$ is viewed as a graph as in Remark \ref{rem: X as a cover}) is a $|G|$-cover of $\Gamma(G)$.

\end{enumerate}
Alternatively, one could work with a special model instead of a saturated one to avoid any set-theoretic issues. 
\end{fact}

\begin{remark}\label{rem: unbounded commutator products}
The reason given for (3) in \cite{hodges1993model} appears to assume that $K_X$ is definable, which is not the case, so we provide a proof.

Assume that $C$ is an infinite nice graph. 
%
%
%
Given $m,n \in \omega$, let $\psi_{m,n}$ be the sentence
$$ \forall g_1, \ldots, g_{n} \in Z \exists x \in Z \left( \forall z_1, \ldots, z_{2m} \bigwedge \left(g^{\alpha_1}_1 \cdot \ldots \cdot g_{n}^{\alpha_{n}} \cdot x^\alpha \neq \prod_{j=1}^{m}[z_{2j-1}, z_{2j}]  \right) \right),$$
where the conjunction is over all $(\alpha_1, \ldots, \alpha_{n-1}, \alpha) \in 
\{ 0, \ldots, p-1 \}^{n+1}$ with $\alpha \neq 0$.

We claim that $G(C) \models \psi_{m,n}$ for all $m,n \in \omega$. Indeed, by \cite[Lemma A.3.2(b)+(d)]{hodges1993model}, a basis for $G(C)'=Z(G(C))$ is given by the set of all elements of the form $[a,b]$ with $a,b$ elements in $C$ not connected by an edge. Since $C$ is nice, given any vertices $v \neq w$ there is at most one vertex to which both $v$ and $w$ are connected. As $C$ is infinite, we can find  pairwise distinct elements $a$ and $\{ b_i : i < \omega \}$ in $C$ such that $a$ is not connected by an edge to any of the $b_i$'s. Indeed, take any pairwise distinct elements $a', a'', (c_i : i < \omega)$ in $C$. For each $i$, at most one of $a', a''$ is connected to $c_i$, hence for one of $a,a'$ there are infinitely many $i<\omega$ such that it is not connected to $c_i$.
Hence $\{[a,b_i] : i < \omega \}$ is a linearly independent set and $Z(G(C))$ has infinite dimension. Now, given arbitrary $m,n$ and $g_i \in Z(G(C))$, assume that each $g_i$ is a linear combination of at most $t$ vectors from the basis. If we take $x$ to be a product of more than $nt + m$ non-trivial commutators, the equality above cannot hold (as $\alpha \neq 0$, and for any $l \in \omega$ the sum of $l$ vectors in a basis cannot be written as a linear combination of less than $l$ vectors in the same basis).

Now let $G$ be a saturated model of $\Th(G(C))$, $\lambda := |G|$. In view of Remark \ref{rem: invariance of tranversal}, it is enough to find a sequence $(h_i : i <\lambda)$ of elements in $Z(G)$ which are linearly independent in the elementary abelian $p$-group $Z(G)$ over $G'$. This can be done by transfinite induction using saturation of $G(C)$ and the fact that given any finitely many elements $g_1, \ldots, g_n \in Z(G)$, we can find an element $x \in Z(G)$ independent from them over $G'$ since $G \models \psi_{m,n}$ for all $m \in \omega$ (see Definition \ref{def: indep over derived subgp}).
 
\end{remark}


\begin{remark}\label{rem: cover explanation}
We provide some details concerning Fact \ref{fac: decomposition}(6).
The fact that this is a cover in any model of $\Th(G(C))$ doesn't seem to follow explicitly from the statement of Axiom 10, as stated in the proof of \cite[Corollary A.3.14(c)]{hodges1993model}. However, it follows by an argument similar to the proof of Remark \ref{rem: unbounded commutator products}. This time, working in the vector space $ G/Z(G)$ and using \cite[Lemma A.3.9]{hodges1993model}, the proof of Axiom 10, and \cite[Lemma A.3.2(c)]{hodges1993model} in $G(C)$, all the relevant sets in Definition \ref{def: cover} are infinite, so by saturation can be demonstrated to be of size $|G|$ in a saturated model $G$.  
\end{remark}

The following lemma is a refinement of Fact \ref{fac: decomposition}(4) and \cite[Lemma 4.12]{baudisch2002mekler}.

\begin{lemma} \label{Lem_GlueAut}
	Let $G$ be a saturated model of $\Th(G(C))$, $X$ be a transversal, and $K_X \leq Z(G)$ be such that  $G = \langle X \rangle \times K_X$. Let  $Y$ and $Z$ be two small subsets  of $X$ and let $\bar h_1, \bar h_2$ be two tuples in $K_X$. Suppose that 
	\begin{itemize}
		\item there is a bijection $f$ between  $Y$ and $Z$ which respects the $1^{\nu}$-, $p$-, and $1^{\iota}$-parts, the handles, and $\tp_{\Gamma} (Y^{\nu}) = \tp_{\Gamma} (f(Y^{\nu}))$,
		\item $\tp_{K_X}(\bar h_1) = \tp_{K_X}(\bar h_2)$.
		\end{itemize}
Then there is an automorphism of $G$ coinciding with $f$ on $Y$ and sending  $\bar h_1$ to  $\bar h_2$. 
\end{lemma}

\begin{proof}
By Remark \ref{rem: X as a cover}, we identify $\Gamma(G)$ with $X^\nu$. By saturation of $\Gamma(G)$, $f \restriction Y_\nu$ extends to an automorphism $\sigma$ of the graph $X^\nu$. As $X$ is a $|G|$-cover of $X^\nu$ by saturation of $G$, and $f$ respects the $1^{\nu}$-, $p$-, and $1^{\iota}$-parts and the handles, $\sigma$ extends to an automorphism $\tau$ of the graph $X$ agreeing with $f$ using a straightforward back-and-forth argument. By Fact \ref{fac: decomposition}(5), we have that $\langle X \rangle \cong G(X)$ and $\tau$ lifts to an automorphism of the group $G(X)$, hence to an automorphism $\tilde{\tau}$ of $\langle X \rangle$ extending $f$ by construction. As $K_X$ is saturated by Fact \ref{fac: decomposition}(3), there is an automorphism $\rho$ of $K_X$ which maps $\bar h_1$ to $\bar h_2$. We can now take the cartesian product of $\tilde{\tau} \times \rho$ to obtain an automorphism of $G$ which extends $f$ and maps $\bar h_1$ to $\bar h_2$.  
\end{proof}

Next, we observe that in Fact \ref{fac: decomposition} the choice of a transversal and an elementary abelian subgroup of the center in the decomposition of $G$ can be made entirely independently of each other.

\begin{lemma}\label{lem: invariance of tranversal}
Let $G$ be any model of $\Th(G(C))$ and let $X$ be a transversal of $G$. Then we have $G' = \langle X\rangle '$.
\end{lemma}

\proof Let $K$ be a subgroup of $Z(G)$ as in Fact \ref{fac: decomposition}, such that $G = \langle X\rangle \times K$. It is enough to show that for all $g, g' \in G$, we have that $[g,g']$ is in $\langle X\rangle '$. We choose 
$x, y \in \langle X \rangle $ and $h,k \in K$ such that $g =  x \cdot  h$ and $g' = y \cdot  k$. Then, using that $G$ is $2$-nilpotent, we have $$[g,g'] = [  x \cdot h, y \cdot k] = [x, y],$$ which is in $\langle X\rangle '$.
\qed

\begin{remark}\label{rem: invariance of tranversal}
Let $X$ and $Y$ be two transversals, and let $K_X$ be such that $G = \langle X \rangle \times K_X$. Then the proof of \cite[Theorem A.3.14(d)]{hodges1993model} and Lemma \ref{lem: invariance of tranversal} imply that $G = \langle Y \rangle \times K_X$. Furthermore, that proof also implies that for any set $B$ of elements in $Z(G)$ which are linearly independent in  the elementary abelian $p$-group $Z(G)$ over $G'$ (seen as an $\mathbb F_p$-vector space), there is a subgroup $K$ of $G$ containing $B$ such that for any transversal $X$ of $G$, we have that $G=\langle X\rangle \times K$.
	\end{remark}


\begin{defn}\label{def: indep over derived subgp}

For any small cardinal $\kappa$, let $\Psi((y_i)_{ i \in \kappa})$ be the partial type  consisting of the formulas ``$y_i \in Z(G)$'' for all $i < \kappa$, and the formulas   
$$ \forall z_0, \dots, z_{2m-1}
\left(
\bigwedge_{(\alpha_0, \dots , \alpha_{n-1}) \in \{0, \ldots, p-1 \}^{ n}\setminus ( 0, \dots, 0)} 
\left(\, y_{i_0}^{\alpha_0} \cdot \ldots \cdot y_{i_{n-1}}^{\alpha_{n-1}}
 \neq 
 \prod_{j=0}^{m-1} [z_{2j}, z_{2j+1}] \ 
 \right)\right)$$
 
for all $m,n \in \mathbb{N}$ and $i_0, \ldots, i_n \in \kappa$. 

An easy inspection yields that for any tuple $\bar{b}$, $\bar b \models \Psi(\bar y)$ if and only if $\bar{b}$ is a tuple of central elements linearly independent in  the elementary abelian $p$-group $Z(G)$ over $G'$ (seen as an $\mathbb F_p$-vector space).
\end{defn}

Let now $\Phi(\bar x)$, with $\bar x = \bar{x}^{\nu \frown} \bar{x}^{p \frown} \bar{x}^{\iota}$, be the partial type given by Lemma \ref{lem: transversal type def}.
 Consider the partial type $$\pi(\bar x ,\bar{y}) = \Phi(\bar x) \cup  \Psi(\bar y).$$

\begin{prop} \label{prop: type def}
Let $G$ be a model of $\Th(G(C))$. Then $G \models \pi(\bar a ,\bar{b})$ if and only if we can extend $\bar a$ to a transversal $X$ of $G$ and find a subset $H \subseteq Z(G)$ containing $\bar{b}$ which is linearly independent over $G'$, so that $G = \langle X \rangle \times \langle H \rangle$.
\end{prop}
\begin{proof}
Combining this with Remark \ref{rem: invariance of tranversal}, there is a subgroup $K$ of $G$ containing $\bar b$ such that for any transversal $X$ of $G$, we have that $G=\langle X\rangle \times K$.
Combining this with Lemma \ref{lem: transversal type def}, we can conclude.

\end{proof}

\section{Preservation of NIP}\label{sec: NIP}
We begin with the simplest case demonstrating that NIP is preserved. Recall the following basic characterization of NIP.

\begin{fact} (see e.g. \cite{adler2008introduction}) \label{fac: char of NIP}
Let $T$ be a complete first-order theory and let $\mathbb{M} \models T$ be a monster model. Let $\kappa$ be the regular cardinal $ |T|^+$. Then the following are equivalent.
\begin{enumerate}
\item $T$ is NIP.
\item For every indiscernible sequence $I = (\bar a_i : i\in \kappa)$ of finite tuples and a finite tuple $\bar b$ in $\mathbb{M}$, there is some $\alpha < \kappa$ such that $\tp(\bar b \bar a_i) = \tp(\bar b \bar a_j)$ for all $i,j > \alpha$.
\end{enumerate}
\end{fact}

As in Section \ref{sec: Mekler},  let $C$ be a nice graph and let $G(C)$ be the $2$-nilpotent group of exponent $p$ which is freely generated (in the variety of 2-nilpotent groups) by the vertices of $C$ by imposing that two generators commute if and only if they are connected by an edge in $C$.

\begin{theorem} \label{thm: NIP}
$\Th(C)$ is NIP if and only if $\Th(G(C))$ is NIP.
\end{theorem}
\begin{proof}
If $\Th(G(C))$ is NIP, then $\Th(C)$ is also NIP as $C$ is interpretable in $G(C)$. 

Now, we want to prove the converse. Let $G \models \Th(G(C))$ be a saturated model, and assume that $\Th(G(C))$ has IP but $\Th(C)$ is NIP. Fix $\kappa$ to be $ (\aleph_0)^+$. Then there is some formula $\phi(\bar{x},\bar{y}) \in L_G$, and a sequence $I = (\bar{a}_i : i \in \kappa)$ in $G$ shattered by $\phi(\bar{x},\bar{y})$, i.e. such that for every $S \subseteq \kappa$, there is some $\bar{b}_S$ in $G$ satisfying $G \models \phi(\bar{b}_S, \bar{a}_i)$ if and only if $ i \in S$.

Let $X$ be a transversal for $G$ and $H \subseteq Z(G)$ a set of elements linearly independent over $G'$ and such that $G = \langle X \rangle \times \langle H \rangle$. Then for each $i \in \kappa$ we have, slightly abusing notation, $\bar{a}_i = t_i (\bar{x}_i , \bar{h}_i)$ for some $L_G$-term $t_i$ and some finite tuples $\bar{x}_i= \bar{x}_i^{\nu \frown} \bar{x}_i^{p \frown} \bar{x}_i^{\iota}$ from $X$  where $\bar{x}_i^\nu, \bar{x}_i^p, \bar{x}_i^\iota$ list all of the elements  of type $1^\nu, p, 1^\iota$ in $\bar{x}_i$, respectively, and $\bar{h}_i$ from $H$. After adding some elements of type $1^\nu$ to the beginning of the tuple and changing the term $t_i$ accordingly, we may assume that for each $i\in \kappa$ and $j< |\bar x_i^p|$, the handle of the j-$th$ element of  $\bar{x}_i^p$ is the $j$-th element of  $\bar{x}_i^\nu$ (there might be some repetitions of elements of type $1^\nu$ as different elements of type $p$ might have the same handle). As $\kappa > |L_G| + \aleph_0$, passing to a cofinal subsequence and reordering the tuples if necessary, we may assume that:
\begin{enumerate}
\item $t_i = t \in L_G$ and $|\bar{x}_i|$ and $ |\bar{h}_i|$ are constant for all $i \in \kappa$,
\item 
$|\bar{x}_i^\nu|, |\bar{x}_i^p|, |\bar{x}_i^\iota|$ are constant for all $i \in \kappa$.
\end{enumerate}

Consider the $L_G$-formula $\phi'(\bar{x},\bar{y}') = \phi(\bar{x}; t(\bar{y}_1 , \bar{y}_2))$ with $\bar{y}' := \bar{y}_1^{\frown} \bar{y}_2$ and $|\bar y_1|= |\bar x_i|$ and $|\bar y_2|= |\bar h_i|$. Let $\bar{a}'_i := \bar{x}_i^{\frown}\bar{h}_i$. Then the sequence $I' := (\bar{a}'_i : i \in \kappa)$ is  shattered by $\phi'(\bar{x},\bar{y}')$. Note however that $I'$ is generally not indiscernible. 

To fix this, let $J = ((\bar{x}'_i)^\frown \bar{h}'_i  : i \in \kappa)$ be an $L_G$-indiscernible sequence of tuples in $G$ with the same EM-type as $I'$. Then we have:

\begin{enumerate}
\item $J$ is still shattered by $\phi'(\bar{x}, \bar{y}')$,
\item for each $i\in \kappa$ and $j< |x_i^p|$, we have that the handle of the $j$-th element of  $(\bar{x}'_i)^p$ is the $j$-th element of  $(\bar{x}'_i)^\nu$ (since being a handle is a definable condition, see Definition \ref{fact: handle}, and the corresponding property was true on all elements in $I'$).
\item The set of all elements of $G$ appearing in the sequence $(\bar{x}'_i : i \in \kappa)$ still can be extended to some transversal $X'$ of $G$.
\item The set of all elements of $G$ appearing in the sequence $(\bar{h}'_i : i \in \kappa)$ can be extended to some set $H' \subseteq Z(G)$ linearly independent over $G'$ and such that $G = \langle X' \rangle \times \langle H' \rangle $.
\end{enumerate}

The last two conditions hold as the sets of all elements appearing in the sequences $(\bar{x}_i : i\in \kappa)$ and $(\bar{h}_i : i \in \kappa)$ satisfied the respective conditions, these conditions are type-definable by Proposition \ref{prop: type def} and $J$ has the same EM-type as $I'$.

Now let $\bar{b} \in G$ be such that both sets $\{i \in \kappa: G \models \phi'(\bar{b}, \bar{a}'_i)\}$ and $ \{ i \in \kappa :  G \models \neg \phi'(\bar{b}, \bar{a}'_i) \}$ are cofinal in $\kappa$. Then $\bar{b} = s (\bar{z}, \bar{k})$ for some term $s \in L_G$ and some finite tuples $\bar{z}$ in $X'$ and $\bar{k}$ in $H'$. Write $\bar{z} = \bar{z}^{\nu \frown} \bar{z}^{p \frown} \bar{z}^{\iota}$, with $\bar{z}^{\nu}, \bar{z}^p, \bar{z}^{\iota}$ listing the elements of the corresponding types in $\bar{z}$. In the same way as extending $\bar x_i$, we may add elements to the tuple $\bar{z}$ and assume that the handle of the $j$-th element of  $\bar{z}^p$ is the $j$-th element of  $\bar{z}^\nu$. 

Consider all of the elements in $\bar z^\nu$ and $((\bar{x}'_i)^\nu: i \in \kappa )$ as elements in $\Gamma(G)$ --- a saturated model of $\Th(C)$, and note that as $\Gamma(G)$ is interpretable in $G$ we have that the sequence $((\bar{x}'_i)^\nu : i \in \kappa)$ is also indiscernible in $\Gamma(G)$. As $\Th(\Gamma(G))$ is NIP, by Fact  \ref{fac: char of NIP} there is some $\alpha < \kappa$ such that $\tp_{\Gamma}(\bar{z}^\nu, (\bar{x}'_i)^\nu) = \tp_{\Gamma}(\bar{z}^\nu, (\bar{x}'_j)^\nu)$ for all $i,j > \alpha$.
%
Moreover, using indiscernibility of the sequence $(\bar{x}'_i)$ and possibly throwing away finitely many elements from the sequence, we have that 
$$(\bar{x}'_i)^p \cap \bar{z}^p = (\bar{x}'_j)^p \cap \bar{z}^p, (\bar{x}'_i)^\iota \cap \bar{z}^\iota = (\bar{x}'_j)^\iota \cap \bar{z}^\iota\ \mbox{(as tuples)}$$ 
and $\bar{x}'_i \cap \bar{x}'_j$ is constant, for all $i,j \in \kappa$. Thus, for any $i,j > \alpha$, the bijection $\sigma_{i,j}$ sending $\bar{x}'_i \bar{z}$ to  $\bar{x}'_j \bar{z}$ and preserving the order of the elements satisfies:
\begin{enumerate}
\item $\tp_{\Gamma}((\bar{x}'_i)^\nu \bar{z}^\nu) = \tp_{\Gamma}(\sigma_{i,j}((\bar{x}'_i)^\nu \bar{z}^\nu))$,
\item the map $\sigma_{i,j}$ fixes $\bar{z}$,
\item the map $\sigma_{i,j}$ respects the $1^{\nu}$-, $p$- and $1^\iota$-parts and the handles (since the handle of the $j$-th element of  $(\bar{x}'_i)^p$ is the $j$-th element of  $(\bar{x}'_i)^\nu$).
\end{enumerate}

Now consider $\bar{k}$ and $(\bar{h}'_i: i \in \kappa)$ as tuples of elements in $\langle H' \rangle$, which is a model of the stable theory $\Th(\langle H' \rangle)$. Moreover, as $(\bar{h}'_i : i \in \kappa)$ is $L_G$-indiscernible and $\Th(\langle H' \rangle)$ eliminates quantifiers, $(\bar{h}'_i : i \in \kappa)$ is also indiscernible in the sense of $\Th(\langle H' \rangle)$. Hence, by stability, there is some $\beta \in \kappa$ such that $\tp_{\langle H' \rangle}(\bar{k}\bar{h}'_i) = \tp_{\langle H' \rangle}(\bar{k} \bar{h}'_j)$ for all $i,j > \beta$.

Now,  Lemma \ref{Lem_GlueAut} gives us an automorphism of $G$ sending $\bar{x}'_i \bar{h}'_i \bar{z} \bar k$ to $\bar{x}'_j \bar{h}'_j \bar{z} \bar k$, so $\tp_G(\bar{x}'_i \bar{h}'_i/ \bar{z} \bar k) = \tp_G(\bar{x}'_j \bar{h}'_j/ \bar{z} \bar k)$ for all $i, j > \operatorname{max}\{\alpha, \beta\}$. This contradicts the choice of $\bar{b} = s (\bar{z}, \bar{k})$.
\end{proof}

\subsection*{An alternative argument for NIP}
An alternative proof can be provided relying on the previous work of Mekler and set-theoretic absoluteness.

Recall that the \emph{stability spectrum} of a complete theory $T$ is defined as the function
$$ f_T(\kappa) := \sup \{ |S_1(M)| : M \models T, |M| = \kappa \}$$
for all infinite cardinals $\kappa$. Furthermore, for every infinite cardinal $\kappa$, let $$\ded \kappa := \sup \{ \lambda : \textrm{exists a linear order of size }\leq \kappa \textrm{ with } \lambda \textrm{-many cuts} \}.$$
See \cite{chernikov2016number} and \cite[Section 6]{chernikov2016non} for a general discussion of the function $\ded \kappa$ and its connection to NIP. We will only need the following two facts.
\begin{fact}[Shelah \cite{shelah1990classification}]\label{fac: NIP by counting types} Let $T$ be a theory in a countable language.
\begin{enumerate}
\item It $T$ is NIP, then $f_T(\kappa) \leq (\ded \kappa)^{\aleph_0}$ for all infinite cardinals $\kappa$.
\item If $T$ has IP, then $f_T(\kappa) = 2^\kappa$ for all infinite cardinals $\kappa$.
\end{enumerate}
\end{fact}

It is possible that in a model of ZFC, $\ded \kappa = 2^\kappa$ for all infinite cardinals $\kappa$ (e.g. in a model of the Generalized Continuum Hypothesis). However, there are models of ZFC in which these two functions are different.

\begin{fact}[Mitchell \cite{mitchell1972aronszajn}]
For every cardinal $\kappa$ of uncountable cofinality, there is a cardinal preserving Cohen extension such that $(\ded \kappa)^{\aleph_0} < 2^\kappa$.
\end{fact}

In the original paper of Mekler \cite{mekler1981stability} it is demonstrated that if $C$ is a nice graph and $\Th(C)$ is stable, then $\Th(G(C))$ is stable. More precisely, the following result is established (in ZFC).
\begin{fact} \label{fac: stability is preserved}
Let $C$ be a nice graph. Then $f_{\Th(G(C))} (\kappa) \leq f_{\Th(C)}(\kappa) + \kappa$ for all infinite cardinals $\kappa$.
\end{fact}

Finally, note that the property ``$T$ is NIP'' is a finitary formula-by-formula statement, hence set-theoretically absolute. Thus in order to prove Theorem \ref{thm: NIP}, it is enough to prove it in \emph{some} model of ZFC. Working in Mitchell's model for some $\kappa$ of uncountable cofinality (hence $(\ded \kappa)^{\aleph_0} + \kappa < 2^\kappa$), it follows immediately by combining Facts \ref{fac: NIP by counting types} and \ref{fac: stability is preserved}.

\section{Preservation of $k$-dependence}\label{sec: kDep}

We are following the notation from \cite{chernikov2014n}, and begin by recalling some of the facts there.

\begin{defn}\rm \label{def: k-dependence}
A formula $\varphi\left(x;y_{0},\ldots,y_{k-1}\right)$ has the \emph{$k$-independence property} (with respect to a theory $T$), if in some model there is a sequence $\left(a_{0,i},\ldots,a_{k-1,i}\right)_{i\in\omega}$
such that for every $s\subseteq\omega^{k}$ there is $b_{s}$ such
that 
\[
\models\phi\left(b_{s};a_{0,i_{0}},\ldots,a_{k-1,i_{k-1}}\right)\Leftrightarrow\left(i_{0},\ldots,i_{k-1}\right)\in s\mbox{.}
\]
Here $x,y_0, \ldots, y_{k-1}$ are tuples of variables.
Otherwise we say that $\varphi\left(x,y_{0},\ldots,y_{k-1}\right)$ is \emph{$k$-dependent}.
A theory is \emph{$k$-dependent} if it implies that every formula is
$k$-dependent.

\end{defn}

To characterize $k$-dependence in a formula-free way, we have to work with a more complicated form of indiscernibility.

\begin{defn}
Fix a language $L_{\opg}^k=\{R(x_0,\ldots ,x_{k-1}),<, P_0(x),\ldots , P_{k-1}(x)\}$. 
An \emph{ordered $k$-partite hypergraph} is an $L^{k}_{\opg}$-structure $ \mathcal{A} = \left(A; <, R, P_0, \ldots , P_{k-1} \right)$ such that:
\begin{enumerate}
\item
$A$ is the (pairwise disjoint) union $P^{\mathcal{A}}_0 \sqcup\ldots \sqcup P^{\mathcal{A}}_{k-1}$,
\item $R^{\mathcal{A}}$ is a symmetric relation so that if $(a_0,\ldots ,a_{k-1})\in R^{\mathcal{A}}$ then $P_i\cap \{a_0, \ldots, a_{k-1}\}$ is a singleton for every $i<k$,
\item
$<^{\mathcal{A}}$ is a linear ordering on $A$ with $P_0(A)<\ldots <P_{k-1}(A)$.
\end{enumerate}

\end{defn}

\begin{fact} \label{fac: Gnp} \cite[Fact 4.4 + Remark 4.5]{chernikov2014n}
Let $\mathcal{K}$ be the class of all finite ordered $k$-partite hypergraphs. Then $\mathcal{K}$ is a Fra\"{i}ss\'e class, and its limit is called the \emph{ordered $k$-partite random hypergraph}, which we will denote by $G_{k,p}$.  An ordered $k$-partite hypergraph $\mathcal{A}$ is a model of $\Th(G_{k,p})$ if and only if:
\begin{itemize}
\item
$(P_i(A), <)$ is a model of DLO for each $i<k$,
\item
for every $j<k$, any finite disjoint sets 
$A_0,A_1\subset {\prod_{i<k, i\neq j}P_i(A)}$
 and
$b_0<b_1\in P_j(A)$, there is $b_0<b<b_1$ such that: $R(b,\bar{a})$ holds for every $\bar{a} \in A_0$ and $\neg R(b, \bar{a})$ holds for every $\bar{a} \in A_1$.
\end{itemize}
\end{fact}
\noindent We denote by $O_{k,p}$ the reduct of $G_{k,p}$ to the language $L^{k}_{\op} = \{<, P_0(x),\ldots , P_{k-1}(x)\}$.

\begin{defn}\rm
Let $T$ be a theory in the language $L$, and let $\M$ be a monster
model of $T$.
\begin{enumerate}
\item Let $I$ be a structure in the language $L_0$.
We say that $\bar{a}=\left(a_{i}\right)_{i\in I}$, 
with $a_{i}$ a tuple in $\M$, is \emph{$I$-indiscernible} over a set of parameters $C \subseteq \M$ if for
all $n\in\omega$ and all $i_{0},\ldots,i_{n}$ and $j_{0},\ldots,j_{n}$
from $I$ we have:
$$
\qftp_{L_0}\left(i_{0}, \ldots, i_{n}\right)=\qftp_{L_0}\left(j_{0},\ldots,j_{n}\right)
\Rightarrow $$
$$\tp_{L}\left(a_{i_{0}}, \ldots, a_{i_{n}} / C\right)=\tp_{L}\left(a_{j_{0}}, \ldots, a_{j_{n}} / C \right)
.$$
\item For $L_0$-structures $I$ and $J$, we say that $\left(b_{i}\right)_{i\in J}$
is \emph{based on} $\left(a_{i}\right)_{i\in I}$ over a set of parameters $C \subseteq \M$ if for any finite
set $\Delta$ of $L(C)$-formulas,  and for any finite tuple $\left(j_{0},\ldots,j_{n}\right)$
from $J$ there is a tuple $\left(i_{0},\ldots,i_{n}\right)$ from
$I$ such that:

\begin{itemize}
\item $\qftp_{L_0}\left(j_{0},\ldots,j_{n}\right)=\qftp_{L_0}\left(i_{0},\ldots,i_{n}\right)$ and
\item $\tp_{\Delta}\left(b_{j_{0}},\ldots,b_{j_{n}}\right)=\tp_{\Delta}\left(a_{i_{0}},\ldots,a_{i_{n}}\right)$.
\end{itemize}
\end{enumerate}
\end{defn}

The following fact gives a method for finding $G_{k,p}$-indiscernibles using structural Ramsey theory.

\begin{fact}\cite[Corollary 4.8]{chernikov2014n}\label{fac: random hypergraph indiscernibles exist}
Let $C \subseteq \M$ be a small set of parameters.
\begin{enumerate}
\item For any $\bar{a}=\left(a_{g}\right)_{g\in O_{k,p}}$, there is some $\left(b_{g}\right)_{g\in O_{k,p}}$ which is $O_{k,p}$-indiscernible over $C$ and is based on $\bar{a}$ over $C$.
\item For any $\bar{a}=\left(a_{g}\right)_{g\in G_{k,p}}$, there is some $\left(b_{g}\right)_{g\in G_{k,p}}$ which is $G_{k,p}$-indiscernible over $C$ and is based on $\bar{a}$ over $C$.
\end{enumerate}

\end{fact}

\begin{fact} \cite[Proposition 6.3]{chernikov2014n} \label{fac: char of NIP_k by preserving indisc}
Let $T$ be a complete theory and let $\mathbb{M} \models T$ be a monster model. For any $k \in \mathbb{N}$, the following are
equivalent:
\begin{enumerate}
\item $T$ is $k$-dependent.
\item For any $\left(a_{g}\right)_{g\in G_{k,p}}$ and $b$ with $a_g, b$ finite tuples in $\mathbb{M}$,
if $\left(a_{g}\right)_{g\in G_{n,p}}$ is $G_{n,p}$-indiscernible over
$b$ and $O_{k,p}$-indiscernible (over $\emptyset$), then it is $O_{k,p}$-indiscernible over $b$.
\end{enumerate}

\end{fact}

We are ready to prove the main theorem of the section.
\begin{theorem}\label{thm: k-dependence is preserved}
For any $k \in \mathbb{N}$ and a nice graph $C$, $\Th(C)$ is $k$-dependent if and only if $\Th(G(C))$ is $k$-dependent.
\end{theorem}

\begin{proof}
Let $G \models \Th(G(C))$ be a saturated model, let $X$ be a transversal, and  let $H$ be a set in $Z(G)$ which is linearly independent over $G'$  such that $G = \langle X \rangle \times \langle H \rangle$. Moreover, fix $\kappa$ to be $\aleph_0^+$.

As in the NIP case, if $\Th(G(C))$ is $k$-dependent, then $\Th(C)$ is also $k$-dependent as $C$ is interpretable in $G(C)$.

Now suppose that $\Th(C)$ is $k$-dependent but $\Th(G(C))$ has the $k$-independence property witnessed by the formula  $\varphi\left(x;y_{0},\ldots,y_{k-1}\right) \in L_G$.
By compactness we can find a sequence $\left(a_{0,\alpha},\ldots,a_{k-1,\alpha}\right)_{\alpha \in \kappa}$ such that for any $s\subseteq \kappa^{k}$ there is some $b_{s}$ such
that 
\[
\models\phi\left(b_{s};a_{0,\alpha_{0}},\ldots,a_{k-1,\alpha_{k-1}}\right)\Leftrightarrow\left(\alpha_{0},\ldots,\alpha_{k-1}\right)\in s\mbox{.}
\]
By the choice of $X$ and $H$, for each $i < k$ and $\alpha \in \kappa$, there is some term $t_{i, \alpha} \in L_G$ and some finite tuples $\bar{x}_{i, \alpha}$ from $X$ and $\bar{h}_{i, \alpha}$ from $H$ such that $a_{i, \alpha} = t_{i,\alpha}(\bar{x}_{i, \alpha}, \bar{h}_{i,\alpha})$. 
As $\kappa > |L_G| + \aleph_0$, passing to a subsequence of length $\kappa$ for each $i<k$ we may assume that $t_{i, \alpha} = t_i$ and $\bar{x}_{i, \alpha} = \bar{x}_{i, \alpha}^{\nu \frown} \bar{x}_{i, \alpha}^{p \frown} \bar{x}_{i, \alpha}^{\iota}$ with $\bar{x}_{i, \alpha}^{\nu}, \bar{x}_{i, \alpha}^{p}, \bar{x}_{i, \alpha}^{\iota}$ listing all elements of the corresponding type in $\bar{x}_{i, \alpha}$ and $|\bar{x}_{i, \alpha}^{\nu}|, |\bar{x}_{i, \alpha}^{p}|, |\bar{x}_{i, \alpha}^{\iota}|$ constant for all $i<j$ and $\alpha \in \kappa$. 
Moreover, in the same way as in the NIP case, we add the handles of the elements in the tuple $\bar x_{i,\alpha}^p$ to the beginning of $\bar{x}_{i, \alpha}^{\nu}$. Taking $\psi(x; y'_0, \ldots, y'_{k-1}) := \phi(x; t_0(y'_0), \ldots, t_{k-1}(y'_{k-1}))$, we see that the sequence $(\bar{x}^{\frown}_{0,\alpha} \bar{h}_{0,\alpha}, \ldots,  \bar{x}^{\frown}_{k-1,\alpha} \bar{h}_{k-1,\alpha}: \alpha \in \kappa)$ is shattered by $\psi$, i.\ e.\ for each $A \subset \kappa^k $ there is some $\bar b$ such that $G \models \psi(\bar{b}; \bar{x}_{0, i_0}^\frown \bar{h}_{0, i_0}, \ldots, \bar{x}_{k-1, i_{k-1}}^\frown \bar{h}_{k-1, i_{k-1}})$ if and only if $  (i_0, \dots, i_{k-1}) \in A$. 
We define an $L_{\op}$-structure on $\kappa$ by interpreting each of the $P_i, i < k$ as some countable disjoint subsets of $\kappa$, and choosing any ordering isomorphic to $(\mathbb{Q}, <)$ on each of the $P_i$'s. 
We pass to the corresponding subsequences of $(\bar x_{i,\alpha}^{\frown} \bar{h}_{i, \alpha}: \alpha \in \kappa)$, namely for each $i \in k$, we consider the sequence given by $(\bar x_{i,\alpha}^{\frown} \bar{h}_{i, \alpha}: \alpha \in P_i)$.
%
%
Taking these $k$ different sequences together we obtain the sequence $(\bar{x}_g^\frown \bar{h}_g : g \in O_{k,p})$ indexed by $O_{k,p}$. This sequence is shattered in the following sense: for each $A \subset P_0 \times \dots \times P_{k-1} $ there is some $\bar b \in G$ such that $G \models \psi(\bar{b}; \bar{x}_{g_0}^\frown \bar{h}_{g_0}, \ldots, {\bar{x}_{g_{k-1}}}^{\frown} \bar{h}_{g_{k-1}})$ if and only if $  (g_0, \dots, g_{k-1}) \in A$. 


By Fact \ref{fac: random hypergraph indiscernibles exist}(1), let $(\bar{y}_{g}^{\frown} \bar{m}_g : g \in O_{k,p})$ be an $O_{k,p}$-indiscernible in $G$ based on $(\bar{x}_g^\frown \bar{h}_g : g \in O_{k,p})$. Observe that, using Proposition \ref{prop: type def} as in the proof of Theorem \ref{thm: NIP}, we still have:
\begin{enumerate}
\item $(\bar{y}_{g}^{\frown} \bar{m}_g : g \in O_{k,p})$ is shattered by $\psi$,
\item the handle for each $j$th element in the tuple $\bar y_{g}^p$ is the $j$th element of the tuple $\bar y_{g}^\nu$,
\item the set of all elements of $G$ appearing in $(\bar{y}_g : g \in O_{k,p})$ is a partial transversal, hence can be extended to a transversal $Y$ of $G$,
\item the set of all elements of $G$ appearing in $(\bar{m}_g : g \in O_{k,p})$ is still a set of elements in $Z(G)$ linearly independent over $G'$, hence can be extended to a linearly independent set $M$ such that $G = \langle Y \rangle \times \langle M \rangle$.
\end{enumerate}
We can expand $O_{k,p}$ to $G_{k,p}$ (see Fact \ref{fac: Gnp}). As $(\bar{y}_{g}^{\frown} \bar{m}_g : g \in O_{k,p})$ is shattered by $\psi$, we can find an element $b \in G$ such that $G \models \psi(b; \bar{y}_{g_0}^{\frown}  \bar{m}_{g_0}, \ldots, \bar{y}_{g_{k-1}}^{\frown} \bar{m}_{g_{k-1}}) \iff G_{k,p} \models R(g_0, \ldots, g_{k-1})$, for all $g_{i}\in P_{i}$. We can write $b = s(\bar{z}, \bar{\ell})$ for some term $s \in L_G$ and some finite tuples $\bar{z} = \bar{z}^{\nu \frown} \bar{z}^{p \frown} \bar{z}^\iota$ in $Y$ and $\bar{\ell}$ in $M$. As usual, extending $\bar{z}^\nu$ if necessary, we may assume that $\bar{z}$ is closed under handles. Taking $\theta(x'; y'_0, \ldots, y'_{k-1}) := \psi(s(x'); y'_0, \ldots, y'_{k-1})$, we still have that $$G \models \theta(\bar{z}^{\frown}\bar{\ell}; \bar{y}_{g_0}^{\frown}  \bar{m}_{g_0}, \ldots, \bar{y}_{g_{k-1}}^{\frown} \bar{m}_{g_{k-1}}) \iff G_{k,p} \models R(g_0, \ldots, g_{k-1})$$ for all $g_{i}\in P_{i}$.

By Fact \ref{fac: random hypergraph indiscernibles exist}(2), we can find $(\bar{z}_g^{\frown} \bar{\ell}_g : g \in G_{k,p})$ which is $G_{k,p}$-indiscernible over $\bar{z}^{\frown}\bar{\ell}$ and is based on $(\bar{y}_{g}^{\frown} \bar{m}_g : g \in G_{k,p})$ over $\bar{z}^{\frown}\bar{\ell}$. Then   we have:
\begin{enumerate}
\item $G \models \theta(\bar{z}^{\frown}\bar{\ell}; \bar{z}_{g_0}^{\frown}  \bar{\ell}_{g_0}, \ldots, \bar{z}_{g_{k-1}}^{\frown} \bar{\ell}_{g_{k-1}}) \iff G_{k,p} \models R(g_0, \ldots, g_{k-1})$, for all $g_{i}\in P_{i}$;
\item for $\bar{z}_g = \bar{z}_g^{\nu \frown} \bar{z}_g^{p \frown} \bar{z}_g^{\iota}$ we have that:
\begin{itemize}
	\item all of these tuples are of fixed length and list elements of the corresponding type,
	\item  the handle of the $j$-th element of  $\bar{z}_g^p$ is the $j$-th element of  $\bar{z}_g^\nu$;
	\end{itemize}
\item the set of all elements of $G$ appearing in $\bar{z}$ and $(\bar{z}_g : g \in G_{k,p})$ is a partial transversal, hence can be extended to some transversal $Z$ of $G$;
\item the set of all elements of $G$ appearing in $\bar{\ell}$ and $(\bar{\ell}_g : g \in G_{k,p})$ is still a set of elements in $Z(G)$ linearly independent over $G'$, hence can be extended to a linearly independent set $L$ such that $G = \langle Z \rangle \times \langle L \rangle$;
\item $(\bar{z}_g^{\frown} \bar{\ell}_g : g \in G_{k,p})$ is $O_{k,p}$-indiscernible over $\emptyset$ (follows since $(\bar{z}_g^{\frown} \bar{\ell}_g : g \in G_{k,p})$ is based on $(\bar{y}_{g}^{\frown} \bar{m}_g : g \in G_{k,p})$, which was $O_{k,p}$-indiscernible, as in the proof of \cite[Lemma 6.2]{chernikov2014n}).
\end{enumerate}

Consider now all of the elements in $\bar{z}^\nu$ and $(\bar{z}_{g}^{\nu} : g \in G_{k,p})$ as elements in $\Gamma(G)$, a saturated model of $\Th(C)$, and note that as $\Gamma(G)$ is interpretable in $G$, we have that the sequence $(\bar{z}^{\nu}_g : g \in G_{k,p})$ is also $G_{k,p}$-indiscernible over $\bar{z}^{\nu}$ and is $O_{k,p}$-indiscernible over $\emptyset$, both in $\Gamma(G)$. As $\Th(C)$ is $k$-dependent, it follows by Fact \ref{fac: char of NIP_k by preserving indisc} that $(\bar{z}_{g}^{\nu} : g \in G_{k,p})$ is $O_{k,p}$-indiscernible over $\bar{z}^\nu$ in $\Gamma(G)$. 
Hence for any finite tuples $g_0, \ldots, g_n, q_0, \ldots, q_n \in G_{k,p}$ such that $\tp_{L_{\op}^k}(\bar{g})  = \tp_{L_{\op}^k}(\bar{q})$, we have that  $\tp_{\Gamma}(\bar{z}^\nu_{g_0}, \ldots, \bar{z}^\nu_{g_n}/ \bar{z}^\nu )$ is equal to $\tp_{\Gamma}(\bar{z}^\nu_{q_0}, \ldots, \bar{z}^\nu_{q_n}/ \bar{z}^\nu)$. 
Now, using  $O_{k,p}$-indiscernibility and that $\bar z$ is finite, for each $i < k$ there is some finite $\lambda_i \subseteq  P_i $ such that for all $g \neq q \in P_i$ with $g,q > \lambda_i$ (i.e. $g > h$ for every element $h \in \lambda_i$, and the same for $q$) we have
$$ \bar{z}_{g}^p  \cap \bar{z}^p = \bar{z}_{q}^p \cap \bar{z}^p, \bar{z}_{g}^\iota \cap \bar{z}^\iota = \bar{z}_{q}^\iota\cap \bar{z}^\iota\ \mbox{(as tuples)}
$$
and $\bar{z}_{g} \cap \bar{z}_{q}$ is constant. Thus, for any $g_0, \dots ,g_{k-1}, q_0, \dots, q_{k-1}$  such that  $g_i, q_i > \lambda_i$ and $g_i, q_i \in P_i$, we get that mapping $\bar{z}_{g_0}, \dots,\bar{z}_{g_{k-1}},  \bar{z}$ to  $\bar{z}_{q_0}, \dots,\bar{z}_{q_{k-1}},   \bar{z}$ preserving the positions of the elements in the tuples defines a bijection $\sigma_{\bar g,\bar q}$ 
such that:
\begin{enumerate}
\item $\tp_{\Gamma}(\bar{z}_{g_0}^\nu , \dots,\bar{z}_{g_{k-1}}^\nu ,\bar{z}^\nu) = \tp_{\Gamma}(\sigma_{\bar g,\bar q}(\bar{z}_{g_0}^\nu , \dots,\bar{z}_{g_{k-1}}^\nu, \bar{z}^\nu)),$
\item the map $\sigma_{\bar g,\bar q}$ fixes $\bar{z}$,
\item the map $\sigma_{\bar g,\bar q}$ respects the $1^{\nu}$-, $p$- and $1^\iota$-parts and the handles.
\end{enumerate}


Next we consider all the elements in $\bar{\ell}$ and $(\bar{\ell}_g : g \in G_{k,p})$ as elements in $\langle L \rangle$, a saturated model of the stable theory $\Th(\langle L \rangle)$. By quantifier elimination, we still have that $(\bar{\ell}_g : g \in G_{k,p})$ is both $O_{k,p}$-indiscernible and $G_{k,p}$-indiscernible over $\bar{\ell}$ in $\langle L \rangle$. As $\langle L \rangle$ is stable, so in particular $k$-dependent, by Fact \ref{fac: char of NIP_k by preserving indisc}, $(\bar{\ell}_g : g \in G_{k,p})$ is  $O_{k,p}$-indiscernible over $\bar{\ell}$.

Now let $\bar{g}, \bar{q} \in G_{k,p}$ be such that $g_i, q_i > \lambda_i$ and $g_i,q_i \in P_i$ for all $i < k$, and such that $G_{k,p} \models R(g_0, \ldots, g_{k-1}) \land \neg R(q_0, \ldots, q_{k-1})$ holds. Then by the choice of $\bar{z}^{\frown} \bar{\ell}$ we have that $G \models \theta(\bar{z}^{\frown}\bar{\ell}; \bar{z}_{g_0}^{\frown} \bar{\ell}_{g_0}, \ldots, \bar{z}_{g_{k-1}}^{\frown} \bar{\ell}_{g_{k-1}}) \land \neg \theta(\bar{z}^{\frown}\bar{\ell}; \bar{z}_{q_0}^{\frown}  \bar{\ell}_{q_0}, \ldots, \bar{z}_{q_{k-1}}^{\frown} \bar{\ell}_{q_{k-1}})$. On the other hand, combining the last two paragraphs and using Lemma \ref{Lem_GlueAut}, 
we find an automorphism of $G$ sending $(\bar{z}_{g_0}^\frown  \bar{\ell}_{g_0}, \ldots, \bar{z}_{g_{k-1}}^\frown \bar{\ell}_{g_{k-1}})$ to $(\bar{z}_{q_0}^\frown  \bar{\ell}_{q_0}, \ldots, \bar{z}_{q_{k-1}}^\frown  \bar{\ell}_{q_{k-1}})$ and fixing $\bar{z}^{\frown}\bar{\ell}$ --- a contradiction.


\end{proof}

\begin{cor} \label{cor: strictly k-dep groups}

For every $k \geq 2$, there is a strictly $k$-dependent  pure group $G$. Moreover, we can find such a $G$ with a simple theory.
\end{cor}
\begin{proof}
For each $k \geq 2$, let $A_k$ be the random $k$-hypergraph. It is well-known that $\Th(A_k)$ is simple. Moreover, $A_{k}$ is clearly not $(k-1)$-dependent, as witnessed by the edge relation, and it is easy to verify that $A_{k}$ is $k$-dependent (as it eliminates quantifiers and all relation symbols are at most $k$-ary, see e.g. \cite[Proposition 6.5]{chernikov2014n}).

Now $A_k$, as well as any other structure in a finite relational language, is bi-interpretable with some nice graph $C_k$ by \cite[Theorem 5.5.1 + Exercise 5.5.9]{hodges1993model}, so $C_k$ also has all of the aforementioned properties. Then Mekler's construction produces a group $G(C_k)$ with all of the desired properties, by Theorem \ref{thm: k-dependence is preserved} and  preservation of simplicity from \cite{baudisch2002mekler}. 
\end{proof}

This corollary gives first examples of strictly $k$-dependent groups, however many other questions about the existence of strictly $k$-dependent algebraic structures remain.

\begin{problem}
\begin{enumerate}
\item Are there pseudofinite strictly $k$-dependent groups, for $k>2$? 

\item Are there $\aleph_0$-categorical strictly $k$-dependent groups, for $k>2$? 

\end{enumerate}
We note that the strictly $2$-dependent group in Example \ref{ex: extraspecial} is  both pseudofinite and $\aleph_0$-categorical (see \cite[Proposition 3.11]{macpherson2008one} and the discussion around it). However, Mekler's construction does not preserve $\aleph_0$-categoricity in general (this is mentioned in \cite[Introduction]{baudisch2002mekler}), e.g. because the proof in Remark \ref{rem: unbounded commutator products} shows that if $C$ is an infinite nice graph, then in $G(C)$ there are infinitely many pairwise inequivalent formulas $\phi_n(x)$ expressing that $x$ is a product of at most $n$ commutators.
\end{problem}

\begin{problem}
Are there strictly $k$-dependent fields, for any $k \geq 2$? We conjecture that there aren't any with a simple theory. It is proved in \cite{hempel2016n} that any $k$-dependent PAC field is separably closed, and there are no known examples of fields with a simple theory which are not PAC.
\end{problem}

\section{Preservation of $\NTP_2$} \label{sec: NTP2}
We recall the definition of $\NTP_2$ (and refer to \cite{chernikov2014theories} for further details).
\begin{defn} \label{def: NTP2}
\begin{enumerate}
\item A formula $\phi(x,y)$, with $x,y$ tuples of variables, has $\TP_2$ if there is an array $( a_{i,j} : i,j \in \omega)$ of tuples in $\M \models T$ and some $k \in \omega$ such that:
\begin{enumerate}
\item for all $i \in \omega$, the set $\{ \phi(x, a_{i,j}) : j \in \omega \}$ is $k$-inconsistent.
\item for all $f : \omega \to \omega$, the set $\{ \phi(x, a_{i, f(i)}) : i \in \omega \}$ is consistent.
\end{enumerate}

\item A theory $T$ is $\NTP_2$ if no formula has $\TP_2$ relatively to it.
\end{enumerate}
\end{defn}

\begin{remark} \label{rem : 2-incons}\cite[Lemma  3.2]{chernikov2014theories}
If $T$ is not $\NTP_2$, one can find a formula as in Definition \ref{def: NTP2}(1) with $k = 2$.
\end{remark}

We will use the following formula-free characterization of $\NTP_2$ from \cite[Section 1]{chernikov2014theories}.

\begin{fact} \label{fac: char of NTP2} Let $T$ be a theory and $\M \models T$ a monster model. Let $\kappa := |T|^+$. The following are equivalent:
\begin{enumerate}
\item $T$ is $\NTP_2$.
\item For any array $(a_{i,j} : i \in \kappa, j \in \omega)$ of finite tuples with \emph{mutually indiscernible rows} (i.e. for each $i \in \kappa$, the sequence $\bar{a}_i := (a_{i,j} : j \in \omega)$ is indiscernible over $\{ a_{i',j} : i' \in \kappa \setminus \{i \}, j \in \omega \}$) and a finite tuple $b$, there is some $\alpha \in \kappa $ satisfying the following:
for any $i > \alpha$ there is some  $b'$ such that $\tp(b/ a_{i,0}) = \tp(b' / a_{i,0})$ and $\bar{a}_i$ is indiscernible over $b'$.
\end{enumerate}

\end{fact}

The following can be proved using finitary Ramsey theorem and compactness, see \cite[Section 1]{chernikov2014theories} for the details.

\begin{fact}\label{fac: extracting mut ind}

Let $(a_{\alpha,i} : \alpha,i \in \kappa)$ be an array of tuples from $\M \models T$.
Then there is an array $(b_{\alpha,i} : \alpha, i \in \kappa)$ with mutually indiscernible rows \emph{based on $(a_{\alpha,i} : \alpha,i \in \kappa)$}, i.e. such that for every finite set of formulas $\Delta$, any $\alpha_0, \ldots, \alpha_{n-1} \in \kappa$ and any strictly increasing finite tuples $\bar{j}_0, \ldots, \bar{j}_{n-1}$ from $\kappa$, there are some strictly increasing tuples $\bar{i}_0, \ldots, \bar{i}_{n-1}$ from $\kappa$ such that 
$$\models \Delta( (b_{\alpha_0,i} : i \in \bar{j}_0), \ldots, (b_{\alpha_{n-1}, i} : i \in \bar{j}_{n-1}) ) \iff$$
$$\models \Delta( (a_{\alpha_0,i} : i \in \bar{i}_0), \ldots, (a_{\alpha_{n-1}, i} : i \in \bar{i}_{n-1}) ).$$
\end{fact}

\begin{remark} \label{rem: extraction preserves witness TP2}
If $\phi(x,y)$ and $(a_{\alpha,i} : \alpha,i \in \kappa)$ satisfy the condition in Definition \ref{def: NTP2}(1) and $(b_{\alpha,i} : \alpha, i \in \kappa)$ is based on it, then $\phi(x,y)$ and $(b_{\alpha,i} : \alpha, i \in \kappa)$ also satisfy the condition in Definition \ref{def: NTP2}(1).
\end{remark}

\begin{theorem}\label{thm: NTP2 is preserved}
For any nice graph $C$, we have that $\Th(G(C))$ is $\NTP_2$ if and only if $\Th(C)$ is $\NTP_2$.
\end{theorem}
\begin{proof}
	
	As before, let $G \models \Th(G(C))$ be a monster model, let $X$ be a transversal, and  let $H$ be a set in $Z(G)$ which is linearly independent over $G'$ such that $G = \langle X \rangle \times \langle H \rangle$. Moreover, fix $\kappa$ to be $\aleph_0^+$. If $\Th(G(C))$ is $\NTP_2$ then $\Th(C)$ is also $\NTP_2$ as $C$ is interpretable in $G(C)$.

	Now suppose that $\Th(C)$ is $\NTP_2$, but $\Th(G(C))$ has $\TP_2$. By compactness and Remark \ref{rem : 2-incons} we can find some formula $\phi(x,y)$ and an array $(\bar a_{i,j} : i,j \in \kappa)$ of tuples in $G$ witnessing $\TP_2$ as in Definition \ref{def: NTP2}(1) with $k=2$. 
	Then for all $i,j \in \kappa$ we have $\bar a_{i,j} = t_{i,j}(\bar{x}_{i,j}, \bar{h}_{i,j})$ for some terms $t_{i,j} \in L_G$ and some finite tuples $\bar{x}_{i,j}$ from $X$ and $\bar{h}_{i,j}$ from $H$.

As $\kappa > |L_G| + \aleph_0$, passing to a subsequence of each row, and then to a subsequence of the rows, we may assume that $t_{i,j} = t \in L_G$ and $\bar{x}_{i,j} = \bar{x}_{i,j}^{\nu \frown} \bar{x}_{i,j}^{p \frown} \bar{x}_{i,j}^{\iota}$ with $|\bar{x}_{i,j}^\nu|, |\bar{x}_{i,j}^p|, |\bar{x}_{i,j}^\iota|, |\bar{h}_{i,j}|$ constant for all $i,j \in \kappa$. Again as in the NIP case, we add the handles of the elements in the tuple $\bar x_{i,\alpha}^p$ to the beginning of $\bar{x}_{i, \alpha}^{\nu}$ for all $i,j \in \kappa$. Taking $\psi(x,y') := \phi(x, t(y') )$ with $|y'| = |\bar{x}_{i,j}^\frown \bar{h}_{i,j}|$ and $\bar{b}_{i,j} :=  \bar{x}_{i,j}^\frown \bar{h}_{i,j}$, we have that $\psi(x,y') \in L_G$ and the array $(\bar{b}_{i,j} : i,j \in \kappa)$ still satisfy the condition in Definition \ref{def: NTP2}(1) with $k=2$. 

By Fact \ref{fac: extracting mut ind}, let $(\bar{c}_{i,j} : i,j \in \kappa)$ with $\bar{c}_{i,j} = \bar{y}_{i,j}^\frown \bar{m}_{i,j}$ be an array with mutually indiscernible rows based on $(\bar{b}_{i,j} : i,j \in \kappa)$. Then, arguing as in the proofs of Theorems \ref{thm: NIP} and  \ref{thm: k-dependence is preserved} using type-definability of the relevant properties from Proposition \ref{prop: type def} and Remark \ref{rem: extraction preserves witness TP2}, we have:
\begin{enumerate}
\item $\psi(x,y')$ and the array $(\bar{c}_{i,j} : i,j \in \kappa)$ satisfy the condition in Definition \ref{def: NTP2}(1) with $k=2$;
\item For $\bar{y}_{i,j}= \bar{y}_{i,j}^{\nu \frown} \bar{y}_{i,j}^{p \frown} \bar{y}_{i,j}^{\iota}$ we have that:
\begin{itemize}
	\item all of these tuples are of fixed length and list elements of the corresponding type,
	\item  the handle of the $n$-th element of  $\bar{y}_{i,j}^p$ is the $n$-th element of  $\bar{y}_{i,j}^\nu$;
	\end{itemize}
\item the set of all elements of $G$ appearing in $(\bar{y}_{i,j} : i,j \in \kappa)$ is a partial transversal of $G$ and can be extended to a transversal $Y$ of $G$;
\item the set of all elements of $G$ appearing in $(\bar{m}_{i,j} : i,j \in \kappa)$ is a set of elements in $Z(G)$ linearly independent over $G'$, hence can be extended to a set of generators $M$ such that $G = \langle Y \rangle \times \langle M \rangle$.
\end{enumerate}

Let now $\bar{b}$ be a tuple in $G$ such that $G \models \{ \psi(\bar b, \bar{c}_{i,0}) : i \in \kappa\}$. We have that $\bar b =
s(\bar{y}, \bar{m})$ for some term $s \in L_G$ and some finite tuples $\bar{y}$ in $Y$ and $\bar{m}$ in $M$. Let
$\bar{y} = \bar{y}^{\nu \frown} \bar{y}^{p\frown} \bar{y}^{\iota}$, each listing the elements of the corresponding type. In
the same way as for each of the $\bar y_{i,j}$'s, we add the handles of the elements in the tuple $\bar y^p$ to the
beginning of $\bar{y}^{\nu}$ so that the handle of the $n$-th element of $\bar{y}^p$ is the $n$-th element of
$\bar{y}^\nu$.
%
%
Taking $\theta(x',y') := \psi(s(x'), y')$, we still have that $\bar{y}^\frown \bar{m} \models \{ \theta(x', \bar{c}_{i,0}) : i \in \kappa  \}$ and the set of formulas $\{ \theta(x', \bar{c}_{i,j}) : j \in \kappa \}$ is $2$-inconsistent for each $i \in \kappa$. Moreover, after possibly throwing away finitely many rows, we may assume that 
the rows are mutually indiscernible over $\bar{y}^\frown \bar{m} \cap \bigcup \{\bar{c}_{i,0} : i \in \kappa \} $ (if an element of $\bar{y}^\frown \bar{m}$ appears in $\bar{c}_{i,0}$, then the rows of the array $(\bar{c}_{i',j} : i' \in \kappa, i'\neq i, j \in \kappa)$ are mutually indiscernible over it). This implies that if $z \in \bar{y} \cap \bar{y}_{i,0}$ for some $i$ and $z$ is the $n$-th element in the tuple $\bar{y}_{i,0} $, then it is the $n$-th element in any tuple $\bar{y}_{j,0} $ with $j \in \kappa$.

Consider all of the elements in $\bar{y}^{\nu}$ and $(\bar{y}_{i,j}^{\nu} : i,j \in \kappa)$ as elements in $\Gamma(G)$, a saturated model of $\Th(C)$, and note that as $\Gamma(G)$ is interpretable in $G$ we have that the array $(\bar{y}_{i,j}^\nu : i,j \in \kappa)$ has mutually indiscernible rows in $\Gamma(G)$. As $\Th(\Gamma(G))$ is $\NTP_2$, it follows by Fact \ref{fac: char of NTP2}  that there is some $\alpha \in \kappa$ such that for each $i> \alpha$ there is some tuple $\bar{y}^{\prime \nu}$ such that $\tp_{\Gamma}(\bar{y}^{\nu}/ \bar{y}^{\nu}_{i,0}) = \tp_{\Gamma}(\bar{y}^{\prime \nu} / \bar{y}^{\nu}_{i,0})$ and the sequence $(\bar{y}_{i,j}^{\nu} : j \in \kappa)$ is $L_{\Gamma}$-indiscernible over $\bar{y}^{\prime \nu}$, i.\ e.\  $\tp_{\Gamma}(\bar{y}^{\nu}, \bar{y}^{\nu}_{i,0}) = \tp_{\Gamma}(\bar{y}^{\prime \nu}, \bar{y}^{\nu}_{i,0})$.
Let   $\sigma_0$ be the bijection which maps   $ \bar{y}^{\nu \frown} \bar{y}_{i,0}$ to $\bar{y}^{\prime \nu \frown} \bar{y}_{i,0}$.
Now we want to extend this bijection $\sigma_0$ to a bijection $\sigma$ defined on each element of  $\bar y^{\frown} \bar{y}_{i,0}$ in a type and handle preserving way. To do so, we have to choose an image for each element in $\bar y^{p^\frown} \bar y^\iota$. Let $z$ be the $n$-th element of $ \bar y^p$ and let $u$ be the $n$-th element of $\bar y^\nu$ (i.\ e.\ the handle of $z$).
\begin{itemize}
	\item If $z\not \in  \bar y_{i,0}^p$, then choose $\sigma(z)$ to be any element in $Y^p$ which has handle $\sigma_0(u)$ and is not contained in $ \bar y_{i,0}^p$ (as $Y$ is a $|G|$-cover and $\tp_\Gamma(u) = \tp_\Gamma(\sigma_0(u))$, using Fact \ref{fac: decomposition}(4) there must be infinitely many elements in $Y^p$ for which $\sigma_0(u)$ is a handle, so we can choose one of them which is not contained in the finite tuple $\bar{y}^p_{i,0}$).
	\item If $z \in  \bar y_{i,0}^p$, then we have that $\sigma_0$ fixes $z$ as well as the handle $u$ of $z$. In this case let $\sigma(z)$ be equal to $z$.
	\end{itemize}
Finally, we define $\sigma$ on each element of $\bar y^\iota$ as the identity map. Let $\bar y'= \bar y^{\prime \nu \frown} \sigma(\bar{y}^p)^{\frown} \bar{y}^\iota $. Then we have that  for all $y \in \bar{y}^{\frown} \bar{y}_{i,0}$:
\begin{enumerate}
	\item $\sigma$ is well defined;
\item $\sigma$ fixes all elements in $\bar{y}_{i,0}$;
\item $\sigma$ respects types and handles by construction;
\item $\tp_{\Gamma}(\bar{y}^{\nu}, \bar{y}^{\nu}_{i,0}) = \tp_{\Gamma}(\sigma(\bar{y}^{\nu}, \bar{y}^{\nu}_{i,0}))$ as $\sigma(y) = \sigma_0(y)$ for all $y \in \bar{y}^{\nu \frown} \bar{y}^\nu_{i,0}$.
\end{enumerate}
%

Now consider $\bar{m}$ and $(\bar{m}_{i,j} : i,j \in \kappa)$ as tuples of elements in $\langle M \rangle$, which is a model of the stable theory $\Th(\langle M\rangle)$. Moreover, as $(\bar{m}_{i,j} : i,j \in \kappa)$ has $L_G$-mutually indiscernible rows and $\Th(\langle M \rangle)$ eliminates quantifiers, $(\bar{m}_{i,j} : i,j \in \kappa)$ has mutually indiscernible rows in the sense of $\Th(\langle M \rangle)$. Hence, by Fact \ref{fac: char of NTP2} again, there is some $\beta \in \kappa$ such that for each $i > \beta$ there is some $\tau \in \Aut(\langle M \rangle)$ fixing $\bar{m}_{i,0}$ and such that $(\bar{m}_{i,j} : j \in \kappa)$ is indiscernible over $\bar{m}' := \tau( \bar{m})$. 

Fix some $i > \max\{\alpha, \beta\}$ and let $\bar y'$ and  $\bar m'$ be chosen as above. Then by Lemma \ref{Lem_GlueAut} we find an automorphism of $G$ which maps  $\bar{y} \bar{m}^{\frown} \bar{y}_{i,0} \bar{m}_{i,0}$ to $\bar{y}' (\bar{m}')^{\frown} \bar{y}_{i,0} \bar{m}_{i,0}$, hence 
$$\tp_G(\bar{y}' \bar{m}' / \bar{y}_{i,0} \bar{m}_{i,0}) = \tp_G(\bar{y} \bar{m} / \bar{y}_{i,0} \bar{m}_{i,0}).$$ 
In particular, $G \models \theta(\bar{y}' \bar{m}' , \bar{y}_{i,0} \bar{m}_{i,0})$. We will show that
$$\tp_G( \bar{y}_{i,0} \bar{m}_{i,0}/ \bar{y}' \bar{m}') = \tp_G(\bar{y}_{i,1} \bar{m}_{i,1}/ \bar{y}' \bar{m}'),$$
 which would then contradict the assumption that $\{ \theta(x', \bar{y}_{i,j} \bar{m}_{i,j}) : j \in \kappa \}$ is $2$-inconsistent.

We show that sending    $ \bar y' \bar{y}_{i,0}$ to $ \bar y'\bar{y}_{i,1}$ is a well-defined bijection $f_0$. The only thing to check is that if the $n$-th element $z$ of $\bar y_{i,0}$ is an element of $\bar y'$, then the $n$-th element of $\bar y_{i,1}$ is equal to $z$. This is true as by construction  we have that the sequence $(\bar y_{i,j}: j\in \kappa)$ is indiscernible over $\bar y'\cap  \bar y_{i,0}$ (as $\bar y'\cap  \bar y_{i,0} = \bar y \cap  \bar y_{i,0}$ by construction, and $(\bar y_{i,j}: j\in \kappa)$ is indiscernible over $\bar y \cap  \bar y_{i,0}$ by assumption). Moreover, we have the following properties for $f_0$:
\begin{enumerate}
\item $f_0$ fixes all elements in $\bar{y}'$ (by construction);
\item $f_0$ respects types and handles (by construction);
\item $\tp_{\Gamma}(\bar{y}^{\prime \nu}, \bar{y}^{\nu}_{i,0}) = \tp_{\Gamma}(f_0(\bar{y}^{\prime \nu}, \bar{y}^{\nu}_{i,0}))$ (since by the choice of $\bar{y}^{\prime \nu}$ above, we have that $(\bar{y}^\nu_{i,j}: j\in \kappa)$ is indiscernible over $\bar{y}^{\prime \nu}$ in $\Gamma(G)$).
\end{enumerate}


Similarly, by the choice of $\bar{m}'$ above, the sequence $(\bar{m}_{i,j} : j \in \kappa)$ is indiscernible over $\bar{m}'$, so $\tp_{\langle M \rangle}(\bar{m}_{i,0}, \bar {m}') = \tp_{\langle M \rangle}(\bar{m}_{i,1}, \bar {m}') $

Again, Lemma \ref{Lem_GlueAut} gives us an automorphism of $G$ sending $\bar{y}_{i,0} \bar{m}_{i,0}$ to $\bar{y}_{i,1} \bar{m}_{i,1}$ and fixing $\bar{y}' \bar{m}'$, as wanted.


\end{proof}

\begin{remark}
A slight modification of the same proof shows that $\Th(G(C))$ is \emph{strong} if and only if $\Th(C)$ is strong (see \cite[Sections 2 and 3]{chernikov2014theories} for the relevant definitions).
\end{remark} 

However, since in the proof we have to throw away a finite, but unknown number of rows, this leaves the following problem. 
\begin{problem}
Assume that $\Th(C)$ is of finite burden. Does $\Th(G(C))$ also have to be of finite burden?
\end{problem}

Finally, it would be interesting to investigate what other properties from generalized stability are preserved by Mekler's construction. For example:

\begin{conj}
If $\Th(C)$ is NSOP$_1$ then $\Th(G(C))$ is also NSOP$_1$.\end{conj}

We expect that this could be verified using the methods of this paper and the criterion from \cite{chernikov2016model} and \cite{kaplan2017kim}.

\bibliography{refs}
\end{document}